\documentclass{amsart}

\usepackage[left=1.5in, right=1.5in]{geometry}     
\usepackage{graphicx}              
\usepackage{amsmath}               
\usepackage{amsfonts}              
\usepackage{amsthm}                
\usepackage{enumitem}
\usepackage{amssymb}
\usepackage{url}
\usepackage{color}
\usepackage{mathrsfs}
\usepackage{mathtools}
\usepackage{mathabx}
\usepackage{dsfont}

\numberwithin{equation}{section}

\newtheorem{thm}{Theorem}[section]
\newtheorem{lem}[thm]{Lemma}
\newtheorem{prop}[thm]{Proposition}
\newtheorem{cor}[thm]{Corollary}

\theoremstyle{definition}
\theoremstyle{definition}\newtheorem{example}[thm]{Example}
\theoremstyle{definition}\newtheorem*{remark}{Remark}

\newcommand{\norm}[1]{\left\lVert#1\right\rVert} 
\newcommand{\RR}{\mathbb{R}}

\newcommand{\NN}{\mathbb{N}}

\newcommand{\Sumn}{\displaystyle\sum}
\newcommand{\Suma}{\Sumn_{\alpha\in R_+}}
\newcommand*\diff{\mathop{}\!\mathrm{d}}

\newcommand{\al}{\alpha}
\newcommand{\alx}{\langle\alpha,x\rangle}

\newcommand{\IntN}{\displaystyle\int_{\RR^N}}
\newcommand{\1}{\mathds{1}}
\newcommand\restr[2]{{
  \left.\kern-\nulldelimiterspace  #1 
  \vphantom{\big|} 
  \right|_{#2} 
  }}

\begin{document}

\title[Logarithmic Sobolev inequalities for Dunkl operators]{Logarithmic Sobolev inequalities for Dunkl operators with applications to functional inequalities for singular Boltzmann-Gibbs measures}
\author{Andrei Velicu}
\address{Andrei Velicu, Department of Mathematics, Imperial College London, Huxley Building, 180 Queen's Gate, London SW7 2AZ, United Kingdom \textit{and} Institut de Math\'ematiques de Toulouse, Universit\'e Paul Sabatier, 118 route de Narbonne, 31062 Toulouse, France}
\email{andrei.velicu@math.univ-toulouse.fr}
\date{}


\begin{abstract}
In this paper we study several inequalities of log-Sobolev type for Dunkl operators. After proving an equivalent of the classical inequality for the usual Dunkl measure $\mu_k$, we also study a number of inequalities for probability measures of Boltzmann type of the form $e^{-|x|^p}\diff\mu_k$. These are obtained using the method of $U$-bounds. Poincar\'e inequalities are obtained as consequences of the log-Sobolev inequality. The connection between Poincar\'e and log-Sobolev inequalities is further examined, obtaining in particular tight log-Sobolev inequalities. Finally, we study application to exponential integrability and to functional inequalities for a class of singular Boltzmann-Gibbs measures.
\end{abstract}

\maketitle

\section{Introduction}

The logarithmic Sobolev inequality (or log-Sobolev inequality, for short), on a general measure space $(\Omega, \mathcal{F}, \mu)$ with a quadratic form $Q$ defined on a suitable space of functions on $\Omega$, states that
\begin{equation} \label{euclideanlogsobolev}
\int_\Omega f^2 \log \frac{f^2}{\int_\Omega f^2 \diff \mu} \diff \mu
\leq C Q(f) + D \int_\Omega f^2 \diff \mu,
\end{equation}
for some constants $C$ and $D$. If $D=0$, we say that (\ref{euclideanlogsobolev}) is a tight log-Sobolev inequality. Although this inequality was used before, it was first explicitly recognised in Gross's seminal paper \cite{Gross}. His main result was the equivalence of log-Sobolev inequalities to hypercontractivity. For more information about the properties and uses of log-Sobolev inequalities, as well as some historical background, see \cite{BGL} and \cite{GZ} and references therein.

Dunkl operators are differential-difference operators which generalise the usual partial derivatives by including difference terms defined in terms of a finite reflection group. Although originally introduced to study special functions with certain symmetries, they have found other applications, for example in mathematical physics where they have been used to study Calogero-Moser-Sutherland (CMS) models of interacting particles. A short introduction to the theory of Dunkl operators is given below in Section \ref{SEC:intro}. More information about applications to CMS models can be found in \cite{vDV}, and an overview of their use in probability theory is contained in \cite{GRY}. 

%Dunkl operators satisfy many useful properties, for instance they commute and they satisfy an integration by parts formula for a weighted measure. Moreover, with the help of an intertwining operator which connects Dunkl operators with the usual partial derivatives, one can ultimately define a Dunkl transform which is a generalisation of the usual Fourier transform and with which is shares many essential properties. On the other hand, certain elementary properties fail, most notably the Leibniz rule and the chain rule. This is the main source of difficulty in our study of log-Sobolev inequalities since these two results are prerequisites for virtually all the classical methods. Since no standard recipe applies, we study our results on a case by case basis where the main challenge is to bound the difference terms that arise as a particularity of Dunkl operators.

We begin our study with a log-Sobolev inequality for the Dunkl measure $\mu_k$ which we prove using the Sobolev inequality for Dunkl operators and Jensen's inequality:
\begin{equation} \label{basiclogsob_intro}
\int_{\RR^N} f^2 \log \frac{f^2}{\int f^2 \diff\mu_k} \diff\mu_k 
\leq C_1 \int_{\RR^N} |\nabla_k f|^2 \diff\mu_k + C_2 \int_{\RR^N} f^2 \diff\mu_k.
\end{equation} 
Here $\mu_k$ is the Dunkl measure and $\nabla_k$ is the Dunkl gradient (see Section \ref{SEC:intro} for a definition of these terms and an introduction to Dunkl theory). This result will be the basis of many of the subsequent inequalities. Our main aim is to study functional inequalities for the Boltzmann-type probability measures
\begin{equation} \label{Boltzmann_meas}
\diff\mu_U = \frac{1}{Z} e^{-|x|^p} \diff\mu_k,
\end{equation}
where $Z$ is just a normalising constant. The strategy to prove (non-tight) log-Sobolev inequalities for such measures is to apply inequality \eqref{basiclogsob_intro} to a function with a suitable weight. This will indeed almost prove the inequality we desire, except for a few residual terms. In order to estimate these terms we use $U$-bounds, which were introduced in \cite{HZ} as part of powerful machinery to study quite general functional inequalities.

We also exploit the connection between the log-Sobolev and Poincar\'e inequalities. In general, it is known that the tight log-Sobolev inequality implies the Poincar\'e inequality. On the other hand, a non-tight log-Sobolev inequality, in the presence of a Poincar\'e inequality, can be improved to obtain a tight log-Sobolev inequality. We use these ideas both to produce new Poincar\'e inequalities for Dunkl operators, and to deduce tight log-Sobolev inequalities from our previous results.

Let us summarise our main results. Firstly, for $p>1$ and $\diff\mu_U$ defined by \eqref{Boltzmann_meas}, we shall prove the Poincar\'e inequality
$$\IntN \left| f-\IntN f \diff\mu_U \right|^2 \diff\mu_U \leq C \IntN |\nabla_k f|^2 \diff\mu_U.$$
For the same measures but for $p\geq 2$, we shall prove the tight log-Sobolev inequality
$$\IntN f^2 \log \frac{f^2}{\int f^2 \diff\mu_U} \diff\mu_U 
\leq C \IntN |\nabla_k f|^2 \diff\mu_U.$$
Such an inequality cannot hold for $1<p<2$, and in this range we shall prove a more general tight $\Phi$-Sobolev inequality
$$\IntN \Phi(f^2) \diff\mu_U - \Phi \left( \IntN f^2 \diff\mu_U \right)
\leq C \IntN |\nabla_k f|^2 \diff\mu_U,$$
where $\Phi(x)=x(\log(x+1))^2$ and $s=2\frac{p-1}{p}$. Finally, we also prove a generalised log-Sobolev inequality in $L^1$:
$$\IntN f \left|\log \frac{|f|}{\int |f| \diff\mu_U} \right|^s \diff\mu_U \leq C_1 \IntN |\nabla_k f| \diff\mu_U + C_2 \IntN |f| \diff\mu_U,$$
where $p\geq 1$ and $s=\frac{p-1}{p}$.

In terms of applications, we first prove exponential integrability of Lipschitz functions for probability measures of the form \eqref{Boltzmann_meas}, as well as a Gaussian measure concentration property for the same family of measures.

Finally, we also study applications of our inequalities to singular Boltzmann-Gibbs measures. A good expository paper on functional inequalities for such measures is \cite{CL}. In this paper, the probability measures in question are of the form $\frac{1}{Z} \1_D e^{-U} \diff x$, where $Z$ is a normalising constant, $\1_D$ is the indicator function of $D=\{ x\in \RR^N | x_1>x_2>\ldots >x_N\}$, and 
$$ U(x) = V(x) + \sum_{i<j} W(x_i-x_j).$$
In this notation, $V$ is the confinement potential and it is assumed to be  strongly convex, and $W:(0,\infty) \to \RR$ is the interaction potential, which is assumed to be convex. This setting can be naturally interpreted in terms of Dunkl theory. Indeed, the set $D$ corresponds to a Weyl chamber associated to the root system $A_{N-1}$, and the canonical choice $W(u)=-2k\log u$ produces exactly the Dunkl measure for the same root system:
$$e^{-\sum_{i<j} W(x_i-x_j)}=\prod_{i<j} (x_i-x_j)^{2k} =w_k(x).$$
Using this idea, from our results we obtain functional inequalities similar to those of \cite{CL}, for confinement potentials of the form $V(x)=|x|^p$. Note that in our case $V$ is not strongly-convex, so our results complement those of \cite{CL}. Moreover, our results hold for any root system and so they allow for different new types of interaction potentials; a discussion of examples corresponding to different root systems is contained below.

This paper is organised as follows. After a short introduction to Dunkl theory in Section \ref{SEC:intro}, we prove the main log-Sobolev inequality for the Dunkl measure $\mu_k$ in Section \ref{SEC:classicallogsob}. To illustrate the method of $U$-bounds, in the same section we also prove the log-Sobolev inequality for the Gaussian measure $e^{-|x|^2}\diff\mu_k$. More general $U$-bounds are proved in Section \ref{SEC:ubounds}, which we then apply in Section \ref{SEC:weightedlogsob} to obtain the desired (non-tight) log-Sobolev inequalities for Boltzmann measures. In Section \ref{poincaresection} we prove Poincar\'e inequalities which we then use in Section \ref{SEC:tightlogsob} to obtain tight log-Sobolev inequalities. Finally, in Section \ref{SEC:appl} we discuss applications to exponential integrability and singular Gibbs measures.

\section{Introduction to Dunkl theory}
\label{SEC:intro}

In this section we will present a very quick introduction to Dunkl operators. For more details see the survey papers \cite{Rosler} and \cite{Anker}.

A root system is a finite set $R\subset \RR^N\setminus \{0\}$ such that $R \cap \alpha \RR = \{ -\alpha, \alpha\}$ and $\sigma_\alpha(R) = R$ for all $\alpha\in R$. Here $\sigma_\alpha$ is the reflection in the hyperplane orthogonal to the root $\alpha$, i.e.,
$$ \sigma_\alpha x = x - 2 \frac{\alx}{\langle \alpha,\alpha \rangle} \alpha.$$
The group generated by all the reflections $\sigma_\alpha$ for $\alpha\in R$ is a finite group, and we denote it by $G$. 

%The Weyl chambers associated to the root system $R$ are the connected components of $\RR^N \setminus \{x\in\RR^N : \alx =0 \text{ for some } \alpha\in R \}$. It can be checked that the reflection group $G$ acts simply transitively on the set of Weyl chambers so, in particular, the number of Weyl chambers equals the order of the group, $|G|$.

Let $k:R \to [0,\infty)$ be a $G$-invariant function, i.e., $k(\alpha)=k(g\alpha)$ for all $g\in G$ and all $\alpha\in R$. We will normally write $k_\alpha=k(\alpha)$ as these will be the coefficients in our Dunkl operators. We can write the root system $R$ as a disjoint union $R=R_+\cup (-R_+)$, where $R_+$ and $-R_+$ are separated by a hyperplane through the origin and we call $R_+$ a positive subsystem;  this decomposition is not unique, but the particular choice of positive subsystem does not make a difference in the definitions below because of the $G$-invariance of the coefficients $k$.

The Weyl chambers associated to the root system $R$ are the connected components of $\RR^N \setminus \{x\in\RR^N : \alx =0 \text{ for some } \alpha\in R \}$. It can be checked that the reflection group $G$ acts simply transitively on the set of Weyl chambers so, in particular, the number of Weyl chambers equals the order of the group, $|G|$.

From now on we fix a root system in $\RR^N$ with positive subsystem $R_+$. We also assume without loss of generality that $|\alpha|^2=2$ for all $\al
\in R$. For $i=1,\ldots, N$ we define the Dunkl operator on $C^1(\RR^N)$ by
$$ T_i f(x) = \partial_i f(x) + \Suma k_\alpha \alpha_i \frac{f(x)-f(\sigma_\alpha x)}{\alx}.$$
We will denote by $\nabla_k=(T_1,\ldots, T_N)$ the Dunkl gradient, and $\Delta_k = \displaystyle\sum_{i=1}^N T_i^2$ will denote the Dunkl Laplacian. Note that for $k=0$ Dunkl operators reduce to partial derivatives, and $\nabla_0=\nabla$ and $\Delta_0=\Delta$ are the usual gradient and Laplacian.

We can express the Dunkl Laplacian in terms of the usual gradient and Laplacian using the following formula:
\begin{equation} \label{Dunkllaplacian}
\Delta_k f(x) = \Delta f(x) + 2\Suma k_\alpha \left[ \frac{\langle \nabla f(x),\alpha \rangle}{\alx} - \frac{f(x)-f(\sigma_\alpha x)}{\alx^2} \right].
\end{equation}

The weight function naturally associated to Dunkl operators is
$$ w_k(x) = \prod_{\alpha\in R_+} |\alx|^{2k_\alpha}.$$
This is a homogeneous function of degree
$$ \gamma := \Suma k_\alpha.$$
We will work in spaces $L^p(\mu_k)$, where $\diff\mu_k = w_k(x) \diff x$ is the weighted measure; the norm of these spaces will be written simply $\norm{\cdot}_p$. With respect to this weighted measure we have the integration by parts formula
$$ \IntN T_i(f) g \diff\mu_k = - \IntN f T_i(g) \diff\mu_k.$$

For any $f \in L^1_\text{loc}(\mu_k)$ we say that $T_if$ exists in a weak sense if there exists $g\in L^1_\text{loc}(\mu_k)$ such that
$$ \IntN f T_i\varphi \diff\mu_k = -\IntN g \varphi \diff\mu_k \qquad \forall \varphi \in C_c^\infty(\RR^N)$$
and we write $T_if=g$ (higher order derivatives are defined similarly). We can then define a Dunkl Sobolev space $W_k^{n,p}(\RR^N)$ for all $n\in\NN$ and $1\leq p \leq \infty$ as the space of all functions $f\in L^p(\mu_k)$ for which $T^\eta f$ exists in a weak sense and $T^\eta f \in L^p(\mu_k)$ for all $\eta \in \NN_0^N$ with $|\eta| \leq n$. It can be checked (following, for example, the ideas of \cite[Section 5.2]{Evans}) that this is a Banach space under the norm
$$ \norm{f}_{W_k^{n,p}(\RR^N)} := \left(\sum_{\eta\in \NN_0^N, |\eta| \leq n} \norm{T^\eta f}_p^p \right)^{1/p}.$$
In the particular case $p=2$, we write $H_k^n(\RR^N) := W_k^{n,2}(\RR^N)$. More generally, for any measure $\mu$ we can define $W_k^{n,p}(\mu)$ as the space for which $T^\eta f \in L^p(\mu)$ for all $0\leq |\eta| \leq n$.

One of the main differences between Dunkl operators and classical partial derivatives is that the Leibniz rule does not hold in general. Instead, we have the following. 

\begin{lem} \label{generalisedLeibniz}
If one of the functions $f,g$ is $G$-invariant, then we have the Leibniz rule
$$ T_i(fg) = f T_ig + g T_if.$$
In general, we have
$$ T_i(fg)(x) = T_if(x)g(x) + f(x)T_ig(x) - \Suma k_\alpha \alpha_i \frac{(f(x)-f(\sigma_\alpha x))(g(x)-g(\sigma_\alpha x))}{\alx}.$$
\end{lem}

%An important Dirichlet form estimate for the weighted norms $L^2$ of the Dunkl gradient and the classical gradient was proved in \cite{V} using carr\'e-du-champ methods:

%\begin{lem} [\cite{V}] \label{dunkl-euclidean-dirichletform}
%For all $f\in C_0^1(\RR^N)$ we have
%%
%$$ \IntN |\nabla_k f|^2 \diff\mu_k \geq \IntN |\nabla f|^2 \diff\mu_k.$$
%%
%\end{lem}

A Sobolev inequality is available for the Dunkl gradient (see \cite{V}):

\begin{prop}
Let $1 \leq p < N+2\gamma$ and $q=\frac{p(N+2\gamma)}{N+2\gamma-p}$. Then there exists a constant $C>0$ such that we have the inequality 
$$ \norm{f}_q \leq C \norm{\nabla_k f}_p \qquad \forall f\in W^{1,p}_k(\RR^N).$$
\end{prop}

The theory of Dunkl operators is enriched by the construction of the Dunkl kernel, which acts as a generalisation of the classical exponential function. Using the Dunkl kernel it is then possible to define a Dunkl transform, which generalises the classical Fourier transform, with which it shares many important properties. The approaches in this paper are elementary and do not make any use of these notions, so we will not go into further details here; the interested reader can find a more complete account in the review papers recommended at the beginning of this section.

%An important function associated with Dunkl operators is the Dunkl kernel $E_k(x,y)$, defined on $\CC^N\times\CC^N$, which acts as a generalisation of the exponential and is defined, for fixed $y\in\CC^N$, as the unique solution $Y=E_k(\cdot, y)$ of the equations
%%
%$$ T_iY=y_i Y, \qquad i=1,\ldots N,$$
%%
%which is real analytic on $\RR^N$ and satisfies $Y(0)=1$. Another definition of the Dunkl exponential can be given in terms of the intertwining operator $V_k$ which connects Dunkl operators to usual derivatives via the relation
%%
%$$ T_i V_k = V_k \partial_i.$$
%%
%The Dunkl exponential can then be equivalently defined as
%%
%$$ E_k(x,y) = V_k \left( e^{\langle \cdot, y \rangle} \right) (x).$$
%%
%
%The following growth estimates on $E_k$ are known: for all $x\in\RR^N$, $y\in\CC^N$ and all $\beta\in\ZZ^N_+$ we have
%%
%$$ |\partial_y^\beta E_k(x,y)| \leq |x|^{|\beta|} \displaystyle\max_{g\in G} e^{\text{Re}\langle gx,y\rangle}.$$
%%
%
%It is then possible to define a Dunkl transform on $L^1(\mu_k)$ by
%%
%$$ \mathcal{D}_k(f)(\xi)= \frac{1}{M_k}\Intr_{\RR^N} f(x)E_k(-i\xi,x)\diff \mu_k(x), \qquad \text{ for all } \xi\in\RR^N,$$
%%
%where 
%%
%$$ M_k = \IntN e^{-|x|^2/2} \diff\mu_k(x)$$
%%
%is the Macdonald-Mehta integral. The Dunkl transform extends to an isometric isomorphism of $L^2(\mu_k)$; in particular, the Plancherel formula holds. When $k=0$ the Dunkl transform reduces to the Fourier transform.

\section{The main Log-Sobolev inequalities} \label{SEC:classicallogsob}

To begin with, we have the following Dunkl equivalent of the classical log-Sobolev inequality.

%%%%%%%%%%%%%%%%%%%%%%%%%%%%%%%%%%%%%%%%%%%%%%%%%%%%%%%%%%%%%%%%%%%%%%%%%%%%%%%%%%%%%%%%%%%%%%%%%
\begin{thm} \label{logsobolev}
There exists a constant $c \in \RR$ such that for any $\epsilon >0$ and for any $f\in H^1_k(\RR^N)$ we have
$$ \int_{\RR^N} f^2 \log \frac{f^2}{\int f^2 \diff\mu_k} \diff\mu_k 
\leq \epsilon \int_{\RR^N} |\nabla_k f|^2 \diff\mu_k + C(\epsilon) \int_{\RR^N} f^2 \diff\mu_k,$$
where $C(\epsilon) = \frac{N+2\gamma}{2}(\log\frac{1}{\epsilon} - c)$.
\end{thm}
%%%%%%%%%%%%%%%%%%%%%%%%%%%%%%%%%%%%%%%%%%%%%%%%%%%%%%%%%%%%%%%%%%%%%%%%%%%%%%%%%%%%%%%%%%%%%%%%%

\begin{proof}
Fix $f\in H^1_k(\RR^N)$, $f \neq 0$. Then $\frac{f^2}{\int_{\RR^N} f^2 \diff\mu_k} \diff\mu_k$ is a probability measure, and so by Jensen's inequality we have, for any $\delta>0$, 
\begin{align*}
\int_{\RR^N} f^2 \log \frac{f^2}{\int f^2 \diff\mu_k} \diff\mu_k
&=\frac{1}{\delta} \int_{\RR^N} f^2 \diff\mu_k \cdot \int_{\RR^N} \frac{f^2}{\int f^2 \diff\mu_k} \log \left( \frac{f^2}{\int f^2 \diff\mu_k} \right)^\delta \diff\mu_k
\\
&\leq \frac{1}{\delta} \int_{\RR^N} f^2 \diff\mu_k \cdot \log \int_{\RR^N} \left( \frac{f^2}{\int f^2\diff\mu_k} \right)^{1+\delta} \diff\mu_k 
\\
&= \frac{\delta+1}{\delta} \norm{f}_{2}^2 \log \frac{\norm{f}_{2+2\delta}^2}{\norm{f}_2^2}.
\end{align*}
We then use the elementary inequality 
$$ \log x\leq \epsilon x +  \log \frac{1}{\epsilon} - 1,$$ 
which holds for all $x,\epsilon >0$. Thus
\begin{align*}
\int_{\RR^N} f^2 \log \frac{f^2}{\int f^2 \diff\mu_k} \diff\mu_k
&\leq \frac{\delta+1}{\delta}  \left[\epsilon \norm{f}_{2+2\delta}^2 + (\log \frac{1}{\epsilon} - 1) \norm{f}_2^2 \right].
\end{align*}
Finally, by choosing $\delta>0$ such that $2+2\delta = \frac{2(N+2\gamma)}{N+2\gamma-2}$, we can apply Sobolev's inequality to deduce 
\begin{align*}
\int_{\RR^N} f^2 \log \frac{f^2}{\int f^2 \diff\mu_k} \diff\mu_k
&\leq \frac{\delta+1}{\delta} C_{DS}^2 \epsilon \int_{\RR^N} |\nabla_k f|^2 \diff\mu_k + \frac{\delta+1}{\delta} (\log \frac{1}{\epsilon} - 1) \int_{\RR^N} f^2 \diff\mu_k,
\end{align*}
where $C_{DS}$ is the Sobolev constant. A simple relabelling of the constants finishes the proof.
\end{proof}

Using the general results of \cite{Davies}, from this inequality we can deduce a more general $L^p$ result, as well as the ultracontractivity property. 

\begin{cor}
Let $2<p<\infty$. Then, for any $\epsilon >0$ and for any $f \in C_c^\infty(\RR^N)$ such that $f\geq 0$, we have
$$ \IntN f^p \log \frac{f^p}{\int f^p \diff\mu_k} \diff\mu_k 
\leq \epsilon \IntN \nabla_k f \cdot \nabla_k (f^{p-1}) \diff\mu_k 
+ C\left( \frac{2\epsilon}{p} \right) \IntN f^p \diff\mu_k,$$
where $C(\epsilon)$ is as in the previous Theorem. 
\end{cor}

\begin{proof}
This follows from the previous Theorem and \cite[Lemma 2.2.6]{Davies}.
\end{proof}

Finally, we recover the ultracontractivity property for the Dunkl heat semigroup. This was already established in \cite{V} using properties of the heat kernel; using the new log-Sobolev approach, no a priori bounds on the heat kernel are necessary.

\begin{cor}
The Dunkl heat semigroup $(H_t)_{t\geq 0}$ on $L^2(\mu_k)$ with generator $\Delta_k$ is ultracontractive. More precisely, for all $t>0$ we have
$$ \norm{H_tf}_{\infty} \leq C t^{-\frac{N+2\gamma}{4}} \norm{f}_2,$$
for a constant $C>0$.
\end{cor}

\begin{proof}
This follows from the previous Corollary and \cite[Theorem 2.2.7]{Davies}.
\end{proof}

In what follows, we want to study inequalities for probability measures of the form
$$\diff\mu_U:= \frac{1}{Z} e^{-U} \diff\mu_k,$$
where $U = |x|^p$ and $Z=\IntN e^{-U} \diff\mu_k$. To illustrate the method and to motivate the study of $U$-bounds in the next section, we first consider Gaussian weight in the following Theorem. This result will be further refined and generalised in the next sections, but the main lines of the proof will remain the same. 

%%%%%%%%%%%%%%%%%%%%%%%%%%%%%%%%%%%%%%%%%%%%%%%%%%%%%%%%%%%%%%%%%%%%%%%%%%%%%%%%%%%%%%%%%%%%%%%%%
\begin{thm} \label{Gaussianlog-sob}
Let $U=|x|^2$ and consider the probability measure $\mu_U:=\frac{1}{Z} e^{-U} \mu_k$, where $Z=\IntN e^{-U} \diff\mu_k$. Then, there exist constants $C_1,C_2>0$ such that the following inequality holds for all $f\in H^1_k(\RR^N)$:
$$ \int_{\RR^N} f^2 \log \frac{f^2}{\int f^2 \diff\mu_U} \diff\mu_U
\leq C_1 \int_{\RR^N} |\nabla_k f|^2 \diff\mu_U + C_2 \int_{\RR^N} f^2 \diff\mu_U.$$
\end{thm}
%%%%%%%%%%%%%%%%%%%%%%%%%%%%%%%%%%%%%%%%%%%%%%%%%%%%%%%%%%%%%%%%%%%%%%%%%%%%%%%%%%%%%%%%%%%%%%%%%

\begin{proof}
Plugging $\frac{1}{\sqrt{Z}} fe^{-U/2}$ in Theorem \ref{logsobolev}, we have
\begin{align*}
\int_{\RR^N} f^2 \log \frac{f^2}{\int f^2 \diff\mu_U} \diff\mu_U
&\leq \frac{1}{Z} \epsilon \int_{\RR^N} |\nabla_k(fe^{-U/2})|^2 \diff\mu_k 
\\
&\qquad \qquad
+ (C(\epsilon) + \log Z) \int_{\RR^N} f^2 \diff\mu_U + \int_{\RR^N} f^2 U \diff\mu_U.
\end{align*}
Since $U$ is $G$-invariant, we can use the Leibniz rule to compute
\begin{align*}
\frac{1}{Z} \int_{\RR^N} |\nabla_k(fe^{-U/2})|^2 \diff\mu_k
&=\int_{\RR^N} |\nabla_k f - \frac{1}{2} f \nabla U|^2 \diff\mu_U
\\
&\leq 2 \int_{\RR^N} |\nabla_k f|^2 \diff\mu_U + \frac{1}{2} \int_{\RR^N} f^2 |\nabla U|^2 \diff\mu_U
\\
&=2 \int_{\RR^N} |\nabla_k f|^2 \diff\mu_U + 2 \int_{\RR^N} f^2 U \diff\mu_U,
\end{align*}
where, in the last line, we used the fact that $\nabla U = 2x$, so $|\nabla U|^2 = 4U$. Replacing this inequality in the above, we have
\begin{equation} \label{gaussianlogsob2}
\begin{aligned} 
\int_{\RR^N} f^2 \log \frac{f^2}{\int f^2 \diff\mu_U} \diff\mu_U
&\leq 2\epsilon \int_{\RR^N} |\nabla_k f|^2 \diff\mu_U 
\\
&\qquad \qquad
+ (C(\epsilon) + \log Z) \int_{\RR^N} f^2 + (1+2\epsilon) \IntN f^2U \diff\mu_U.
\end{aligned}
\end{equation}

We now use the identity 
\begin{equation} \label{logsobidentity}
 \nabla_k(fe^{-U/2}) = e^{-U/2}\nabla_k f - \frac{1}{2} f e^{-U/2} \nabla U
\end{equation}
to deduce 
\begin{equation} \label{leibnizgaussian}
\begin{aligned} 
\int_{\RR^N} |\nabla_k f|^2 \diff\mu_U 
&= \frac{1}{Z} \int_{\RR^N} |\nabla_k(fe^{-U/2})|^2 \diff\mu_k 
+ \frac{1}{4}\int_{\RR^N} f^2 |\nabla U|^2 \diff\mu_U 
\\
&\qquad 
+ \frac{1}{Z} \int_{\RR^N} f \nabla U \cdot \nabla_k(fe^{-U/2}) e^{-U/2} \diff\mu_k.
\end{aligned}
\end{equation}
Keeping in mind that $\nabla U = 2x$, this equality implies that
\begin{equation} \label{gaussianlogsob1}
\int_{\RR^N} |\nabla_k f|^2 \diff\mu_U  
\geq \int_{\RR^N} f^2U \diff\mu_U + A,
\end{equation}
where
\begin{align*}
A:= \frac{1}{Z} \int_{\RR^N} f \nabla U \cdot \nabla_k(fe^{-U/2}) e^{-U/2} \diff\mu_k
\end{align*}
We now compute the quantity $A$. Firstly, by integration by parts we have
$$ A=-\sum_{i=1}^N \frac{1}{Z} \int_{\RR^N} f e^{-U/2} T_i(2x_i f e^{-U/2}) \diff\mu_k(x).$$
Using Lemma \ref{generalisedLeibniz}, we have
\begin{equation} \label{leibnizxi}
\begin{aligned}
&T_i(x_ife^{-U/2}) 
\\
&\qquad 
= x_i T_i(f e^{-U/2}) + f e^{-U/2} T_i(x_i)
- \Suma k_\alpha \alpha_i e^{-U(x)/2} \frac{(f(x)-f(\sigma_\alpha x))(x_i - (\sigma_\alpha x)_i)}{\alx}
\\
&\qquad 
= x_i T_i(f e^{-U/2}) + f e^{-U/2} (1+\Suma k_\alpha \alpha_i^2)
- \Suma k_\alpha \alpha_i^2 e^{-U(x)/2} (f(x)-f(\sigma_\alpha x))
\end{aligned}
\end{equation}
Thus
\begin{align*}
A
&= -A - 2(N+2\gamma) \int_{\RR^N} f^2 \diff\mu_U + 4 \Suma k_\alpha \int_{\RR^N} f(x)(f(x)-f(\sigma_\alpha x)) \diff\mu_U(x), 
\end{align*}
and so
$$ A = - N \int_{\RR^N} f^2 \diff\mu_U - 2 \Suma k_\alpha \int_{\RR^N} f(x)f(\sigma_\alpha x) \diff\mu_U(x).$$
Using the elementary inequality $2XY \leq X^2+Y^2$, we obtain
$$ A \geq - (N+2\gamma) \int_{\RR^N} f^2 \diff\mu_U.$$

Replacing this in equation (\ref{gaussianlogsob1}), we obtain
\begin{align} \label{keyineq}
\int_{\RR^N}  f^2 U \diff\mu_U 
\leq \int_{\RR^N} |\nabla_kf|^2 \diff\mu_U + (N+2\gamma) \int_{\RR^N} f^2 \diff\mu_U.
\end{align}
Finally, using this in (\ref{gaussianlogsob2}), we have
$$ \int_{\RR^N} f^2 \log \frac{f^2}{\int f^2 \diff\mu_U} \diff\mu_U
\leq C_1 \int_{\RR^N} |\nabla_k f|^2 \diff\mu_U + C_2 \int_{\RR^N} f^2 \diff\mu_U,$$
for some constants $C_1,C_2>0$, as required.
\end{proof}

\section{U-bounds} \label{SEC:ubounds}

Looking back at the proof of the weighted log-Sobolev inequality in Theorem \ref{Gaussianlog-sob}, we can see that inequality (\ref{keyineq}) was the key element. Inequalities of this form are called U-bounds (cf. \cite{HZ}). In this section we will prove more general U-bounds by adapting our proof slightly, and these will later be used to deduce log-Sobolev inequalities.

%%%%%%%%%%%%%%%%%%%%%%%%%%%%%%%%%%%%%%%%%%%%%%%%%%%%%%%%%%%%%%%%%%%%%%%%%%%%%%%%%%%%%%%%%%%%%%%%%
\begin{prop} \label{Ubounds-2p-prop}
Let $p>1$ and consider the function $U(x)=|x|^p$ and the probability measure $\mu_U:= \frac{1}{Z} e^{-U} \mu_k$, where $Z = \IntN e^{-U} \diff\mu_k$. Then, for any $f\in H_k^1(\mu_U)$, we have the inequality
\begin{equation} \label{Ubounds-2p}
\int_{\RR^N} f^2  |x|^{2(p-1)} \diff\mu_U  
\leq C \int_{\RR^N} |\nabla_k f|^2 \diff\mu_U + D \int_{\RR^N} f^2 \diff\mu_U,
\end{equation}
for some constants $C,D>0$.
\end{prop}
%%%%%%%%%%%%%%%%%%%%%%%%%%%%%%%%%%%%%%%%%%%%%%%%%%%%%%%%%%%%%%%%%%%%%%%%%%%%%%%%%%%%%%%%%%%%%%%%%

\begin{proof}
We follow more closely the proof of (\ref{keyineq}). We have
$$ \nabla U(x) = p|x|^{p-2} x.$$
From (\ref{leibnizgaussian}), we obtain
$$ \frac{p^2}{4} \int_{\RR^N} f^2 |x|^{2(p-1)} \diff\mu_U
\leq \int_{\RR^N} |\nabla_k f|^2 \diff\mu_U - A,$$
where
$$ A= \frac{1}{Z} \int_{\RR^N} f \nabla U \cdot \nabla_k(fe^{-U/2}) e^{-U/2} \diff\mu_k.$$
As before
\begin{align*}
2A
&= -\left[p(p-2)+p(N+2\gamma)-2\gamma p \right] \int_{\RR^N} f^2 |x|^{p-2} \diff\mu_U 
\\
&\qquad
-2p \Suma k_\alpha \int_{\RR^N} |x|^{p-2} f(x)f(\sigma_\alpha x) \diff\mu_U(x)
\\
&\geq -\left[p^2+p(N+2\gamma-2) \right] \int_{\RR^N} f^2 |x|^{p-2} \diff\mu_U.
\end{align*}
Thus
\begin{align} \label{Ubounds-2p-intermediate}
\int_{\RR^N} f^2 |x|^{2(p-1)} \diff\mu_U
\leq \frac{4}{p^2} \int_{\RR^N} |\nabla_k f|^2 \diff\mu_U 
+ \frac{2\left[p^2+p(N+2\gamma-2) \right]}{p^2} \int_{\RR^N} f^2|x|^{p-2} \diff\mu_U.
\end{align}

Assume first that $p>2$ and let $\epsilon>0$. Then, using H\"older's inequality with coefficients $\frac{p}{2(p-1)}$ and $\frac{p-2}{2(p-1)}$, and then Young's inequality with the same coefficients, we have
\begin{align*}
\IntN f^2|x|^{p-2} \diff\mu_U 
&\leq \left( \IntN f^2 \diff\mu_U \right)^\frac{p}{2(p-1)} \left(\IntN f^2 |x|^{2(p-1)} \diff\mu_U\right)^\frac{p-2}{2(p-1)}
\\
&\leq  \frac{p}{2(p-1)} \epsilon^{-\frac{p-2}{p}}\IntN f^2 \diff\mu_U + \frac{p-2}{2(p-1)} \epsilon \IntN f^2 |x|^{2(p-1)} \diff\mu_U.
\end{align*}
Thus, by choosing $\epsilon>0$ small enough such that
$$ 1 > \frac{(p-2)[p^2+p(N+2\gamma-2)]}{p^2(p-1)} \epsilon,$$
we obtain inequality (\ref{Ubounds-2p}) for some constants $C,D>0$.

The case $1<p<2$ requires more care. Let $\phi : \RR^2 \to [0,1]$ be defined by
\begin{equation} \label{aux_fct}
\phi (x) = \begin{cases} 0, & |x| < 1 \\
|x|-1, & 1 \leq |x| \leq 2 \\
1, & |x| > 2.
\end{cases}
\end{equation}
Note that $\phi$ is radial, so $G$-invariant, and $\nabla \phi(x)=\frac{x}{|x|}$ on $1\leq |x| \leq 2$, and it vanishes elsewhere; in particular, $|\nabla \phi| \leq 1$ on $\RR^N$.
Then, writing $f=f\phi + f(1-\phi)$, we have
\begin{equation} \label{Ubound2-specialcase-gen}
\IntN f^2 |x|^{2(p-1)} \diff\mu_U \leq 2 \IntN |f\phi|^2 |x|^{2(p-1)} \diff\mu_U + 2 \IntN |f(1-\phi)|^2 |x|^{2(p-1)} \diff\mu_U,
\end{equation}
and we estimate each of the terms on the right hand side separately. Firstly, by \eqref{Ubounds-2p-intermediate}, we have
\begin{equation*}
\IntN |f\phi|^2 |x|^{2(p-1)} \diff\mu_U 
\leq \frac{4}{p^2} \IntN |\nabla_k (f\phi)|^2 \diff\mu_U + C_p \IntN |f\phi|^2 |x|^{p-2} \diff\mu_U
\end{equation*}
where $C_p:=\frac{2[p^2+p(N+2\gamma-2)]}{p^2}$. By the Leibniz rule (since $\phi$ is $G$-invariant) and using the properties of the function $\phi$, we have 
$$ |\nabla_k(f\phi)|^2 \leq 2 \phi^2 |\nabla_k f|^2 + 2 f^2 |\nabla \phi|^2 
\leq 2 |\nabla_k f|^2 + 2 f^2.$$
Moreover, note that $f\phi=0$ on $|x| \leq 1$ and outside this region we have $|x|^{p-2} \leq 1$ (since $p<2$), so
$$ \IntN |f\phi|^2 |x|^{p-2} \diff\mu_U \leq \IntN f^2 \diff \mu_U.$$
Thus, putting the last three inequalities together, we have obtained
\begin{equation} \label{Ubound2-specialcase-step1}
\IntN |f\phi|^2 |x|^{2(p-1)} \diff \mu_U \leq \frac{8}{p^2} \IntN |\nabla_k f|^2 \diff\mu_U 
+ (\frac{8}{p^2} + C_p) \IntN f^2 \diff\mu_U.
\end{equation}
We now turn to the second term in \eqref{Ubound2-specialcase-gen}. Here we simply note that the function $f(1-\phi)$ is supported on $|x| \leq 2$, and thus
\begin{equation} \label{Ubound2-specialcase-step2}
\IntN |f(1-\phi)|^2 |x|^{2(p-1)} \diff\mu_U \leq 2^{2(p-1)} \IntN |f(1-\phi)|^2 \diff\mu_U \leq 2^{2(p-1)} \IntN f^2 \diff\mu_U.
\end{equation}
Therefore, putting \eqref{Ubound2-specialcase-gen}, \eqref{Ubound2-specialcase-step1} and \eqref{Ubound2-specialcase-step2} together, we obtain the $U$-bound \eqref{Ubounds-2p} for some $C,D>0$.
\end{proof}

From this result we can obtain another type of $U$-bound which will be essential in the later study of log-Sobolev inequalities. Note however that this bound holds in the more restricted range $p\geq 2$.

\begin{cor} \label{Ubounds-2p-cor}
Let $p \geq 2$ and consider the function $U(x)=|x|^p$ and the probability measure $\mu_U:= \frac{1}{Z} e^{-U} \mu_k$, where $Z=\IntN e^{-U} \diff\mu_k$. We then have, for any $f\in H^1_k(\RR^N)$, the inequality
\begin{equation} \label{Ubounds-2p-cor-eqn}
\IntN f^2 |x|^{p} \diff\mu_U \leq C \IntN |\nabla_k f|^2 \diff\mu_k + D \IntN f^2 \diff\mu_U,
\end{equation}
for some constants $C,D>0$.
\end{cor}

\begin{proof}
We employ the same idea as in the last part of the previous result. Namely, let $\phi$ be the function defined by \eqref{aux_fct} and consider the decomposition $f=f\phi + f(1-\phi)$. We have
\begin{equation} \label{Ubounds-cor-gen}
\IntN f^2 |x|^p \diff\mu_U 
\leq 2\IntN |f\phi|^2 |x|^p \diff\mu_U + 2 \IntN |f(1-\phi)|^2 |x|^p\diff\mu_U.
\end{equation}
The function $f\phi$ vanishes on $|x| \leq 1$ and outside this region we have $|x|^p \leq |x|^{2(p-1)}$ (since $p>2$). Thus
$$ \IntN |f\phi|^2 |x|^p \diff\mu_U \leq \IntN |f\phi|^2 |x|^{2(p-1)} \diff\mu_U \leq \IntN f^2 |x|^{2(p-1)} \diff\mu_U.$$
On the other hand, the function $f(1-\phi)$ is supported on $|x| \leq 2$ and thus
$$ \IntN |f(1-\phi)|^2 |x|^p \diff\mu_U \leq 2^p \IntN f^2 \diff\mu_U.$$
Putting these inequalities together, we obtain
$$\IntN f^2 |x|^p \diff\mu_U 
\leq 2 \IntN f^2 |x|^{2(p-1)} \diff\mu_U + 2^{p+1} \IntN f^2 \diff\mu_U.$$
Finally, using inequality \eqref{Ubounds-2p} for the first term on the right hand side, we obtain \eqref{Ubounds-2p-cor-eqn}, as required.
\end{proof}

The two bounds we have proved so far are both in $L^2(\mu_U)$. In the last result of this section we prove an $L^1(\mu_U)$ bound whose proof will require a slightly different approach. 

%%%%%%%%%%%%%%%%%%%%%%%%%%%%%%%%%%%%%%%%%%%%%%%%%%%%%%%%%%%%%%%%%%%%%%%%%%%%%%%%%%%%%%%%%%%%%%%%%
\begin{prop} \label{Ubounds-p}
Let $p>1$ and consider the function $U(x)=|x|^p$ and the probability measure $\mu_U:= \frac{1}{Z} e^{-U} \mu_k$, where $Z = \IntN e^{-U} \diff\mu_k$. We then have, for any $f\in W^{1,1}_k(\RR^N)$, the inequality
$$ \int_{\RR^N} |f| \cdot |x|^{p-1} \diff\mu_U  
\leq C \int_{\RR^N} |\nabla_k f| \diff\mu_U + D \int_{\RR^N} |f| \diff\mu_U,$$
for some constants $C,D>0$.
\end{prop} 
%%%%%%%%%%%%%%%%%%%%%%%%%%%%%%%%%%%%%%%%%%%%%%%%%%%%%%%%%%%%%%%%%%%%%%%%%%%%%%%%%%%%%%%%%%%%%%%%%

\begin{proof}
In order to avoid a singularity that will arise at the origin, we first consider a function $f$ that vanishes on the unit ball. As before, we start with identity (\ref{logsobidentity}). Noting that in this case we have
$$ \nabla U (x) = p |x|^{p-1} \nabla (|x|),$$
the identity above now reads
$$ \nabla_k (fe^{-U}) = e^{-U} \nabla_k f - p f |x|^{p-1} e^{-U} \nabla(|x|).$$
Taking inner product with $\nabla(|x|)$ and integrating on both sides, we have
\begin{equation} \label{logsobgen}
\begin{aligned}
\frac{1}{Z} \int_{\RR^N} & \nabla (|x|) \cdot \nabla_k (fe^{-U}) \diff\mu_k
\\
&\qquad \qquad \qquad
= \int_{\RR^N} \nabla_k(|x|) \cdot \nabla_k f \diff\mu_U - p \int_{\RR^N} |\nabla(|x|)|^2 f |x|^{p-1} \diff\mu_U.
\end{aligned}
\end{equation}
We can use integration by parts on the left hand side to obtain
$$ \frac{1}{Z} \int_{\RR^N} \nabla (|x|) \cdot \nabla_k (fe^{-U}) \diff\mu_k
= - \int_{\RR^N} \Delta_k(|x|) f \diff\mu_U.$$
Replacing this in (\ref{logsobgen}), and using also the fact that $|\nabla(|x|)|=1$ for $x \neq 0$, we have
\begin{align*}
\int_{\RR^N} f|x|^{p-1} \diff\mu_U 
&= \frac{1}{p} \int_{\RR^N} \nabla (|x|) \cdot \nabla_k f \diff\mu_U + \frac{1}{p} \int_{\RR^N} \Delta_k(|x|) f \diff\mu_U
\\
&\leq \frac{1}{p} \int_{\RR^N} |\nabla_k f| \diff\mu_U + \frac{1}{p} \int_{\RR^N} \Delta_k(|x|) f \diff\mu_U.
\end{align*}

Finally, we have
$$ T_i^2(|x|) = T_i \left(\frac{x_i}{|x|} \right) = \frac{1}{|x|} - \frac{x_i^2}{|x|^3} + \Suma k_\alpha \frac{\alpha_i^2}{|x|}, $$
so
$$ \Delta_k(|x|) = (N+2\gamma-1) \frac{1}{|x|}.$$
Therefore, from the above we deduce that (recall that $f$ vanishes on the unit ball)
$$ \int_{\RR^N} f|x|^{p-1} \diff\mu_U  
\leq \frac{1}{p} \int_{\RR^N} |\nabla_k f| \diff\mu_U + \frac{N+2\gamma-1}{p} \int_{\RR^N} |f| \diff\mu_U.$$
Writing $f=f_+-f_-$, where $f_+(x)=\max(f(x),0)$ and $f_-(x)=-\min(f(x),0)$, we can apply this inequality to $f_+$ and $f_-$ separately. Adding the two resulting inequalities, we have
$$ \int_{\RR^N} |f| \cdot |x|^{p-1} \diff\mu_U  
\leq C_1 \int_{\RR^N} |\nabla_k f| \diff\mu_U + D_1 \int_{\RR^N} |f| \diff\mu_U,$$
where $C_1=\frac{1}{p}$ and $D_1=\frac{N+2\gamma-1}{p}$.

Having proved the result for functions that vanish on the unit ball, let us now consider a general function $f\in L^1(\diff\mu_U)$. To prove this more general result, we once again employ the method from the end of the proof of Proposition \ref{Ubounds-2p-prop}. More precisely, let $\phi$ be the function defined in \eqref{aux_fct} and consider $f=\phi f +(1-\phi)f$; the first term vanishes on the unit ball so the above can be applied to it, while the second term has compact support and it will be easy to bound. We have
\begin{align*}
\int_{\RR^N} |f| \cdot |x|^{p-1} \diff\mu_U 
&\leq \int_{\RR^N} |\phi f| \cdot |x|^{p-1} \diff\mu_U + \int_{\RR^N} |(1-\phi) f| \cdot |x|^{p-1} \diff\mu_U
\\
&\leq C_1 \int_{\RR^N} |\nabla_k(\phi f)| \diff\mu_U + D_1 \int_{\RR^N} |\phi f| \diff\mu_U + 2^{p-1} \int_{\RR^N} |f| \diff\mu_U
\\
&\leq C_1 \int_{\RR^N} |\nabla_k f| \diff\mu_U +(C_1+D_1+2^{p-1}) \int_{\RR^N} |f| \diff\mu_U.
\end{align*}
Here, in the last step, we used the Leibniz rule (which holds because $\phi$ is $G$-invariant) and the fact that $|\nabla \phi| \leq 1$. This completes the proof.
\end{proof}
%%%%%%%%%%%%%%%%%%%%%%%%%%%%%%%%%%%%%%%%%%%%%%%%%%%%%%%%%%%%%%%%%%%%%%%%%%%%%%%%%%%%%%%%%%%%%%%%%%%%%%%%%%%%%%%%%%%%%%%%%%%%%%%%%%%%

\section{Log-Sobolev inequalities for Boltzmann-type measures} \label{SEC:weightedlogsob}

In this section we use the $U$-bounds obtained above to deduce log-Sobolev inequalities. To begin with, from Proposition \ref{Ubounds-2p-prop} and Corollary \ref{Ubounds-2p-cor} we obtain a log-Sobolev inequality for probability measures $\diff\mu_U = \frac{1}{Z} e^{-U(x)} \diff \mu_k$, for $U(x)=|x|^p$ and with $ p \geq 2$, and in the range $1\leq p \leq 2$ we prove a $\Phi$-Sobolev inequality. Similarly, from Proposition \ref{Ubounds-p} we obtain a $\Phi$-Sobolev inequality in $L^1(\mu_U)$ for general $1<p<\infty$. The approach in these results is similar to that of Theorem \ref{logsobolev}: first employing Jensen's inequality to take the logarithm outside the integral, and then using the classical Sobolev inequality. $U$-bounds will be used to control residual terms arising from the introduction of a weight.

\begin{thm} \label{logsob>2_thm}
Let $U(x)=|x|^p$ for $p \geq 2$ and consider the probability measure $\diff\mu_U = \frac{1}{Z} e^{-U} \diff\mu_k$, where $Z=\IntN e^{-U} \diff\mu_k$. Then there exist some constants $C_1,C_2>0$ such that for all $f\in H^1_k(\mu_U)$, we have the inequality
\begin{equation} \label{logsob>2} 
\IntN f^2 \log \frac{f^2}{\int f^2 \diff\mu_U} \diff\mu_U 
\leq C_1 \IntN |\nabla_k f|^2 \diff\mu_U
+ C_2 \IntN f^2 \diff\mu_U.
\end{equation}
\end{thm}

\begin{proof}
We apply Theorem \ref{logsobolev} to the function $\frac{1}{\sqrt{Z}} fe^{-U/2}$ and thus we obtain
\begin{align*}
\IntN f^2 \log \frac{f^2}{\int f^2 \diff\mu_U} \diff\mu_U 
&\leq 2\epsilon \IntN |\nabla_k f|^2 \diff\mu_U 
+(C(\epsilon) + \log Z) \IntN f^2 \diff\mu_U
\\
&\qquad
+ \IntN f^2 U \diff\mu_U + 2\epsilon \IntN f^2 |\nabla U|^2 \diff\mu_U.
\end{align*}
Note that in this case we have $|\nabla U|^2 = |x|^{2(p-1)}$ and thus by applying Proposition \ref{Ubounds-2p-prop} and Corollary \ref{Ubounds-2p-cor}, we obtain inequality \eqref{logsob>2} for some constants $C_1,C_2>0$, as required. 
\end{proof}

\begin{thm} \label{weightedlogsobolev1}
Let $U(x)=|x|^p$ for $1 < p \leq 2$ and consider the probability measure $\diff\mu_U = \frac{1}{Z} e^{-U} \diff\mu_k$, where $Z=\IntN e^{-U} \diff\mu_k$. Let $s=2\frac{p-1}{p}$. Then there exist some constants $C_1,C_2>0$ such that for all $f\in H^1_k(\mu_U)$, we have the inequality
$$ \IntN f^2 \left| \log \frac{f^2}{\int f^2 \diff\mu_U} \right|^s \diff\mu_U \leq C_1 \IntN |\nabla_k f|^2 \diff\mu_U + C_2 \IntN f^2 \diff\mu_U.$$
\end{thm}

\begin{proof}
Consider the function $h=\frac{1}{\sqrt{Z}} fe^{-U/2}$. Then $\IntN f^2 \diff\mu_U = \IntN h^2 \diff\mu_k$, and 
\begin{equation} \label{logsob2-1}
\begin{aligned}
\IntN f^2 & \left| \log \frac{f^2}{\int f^2 \diff\mu_U} \right|^s \diff\mu_U
= \IntN h^2 \left| \log \frac{h^2}{\int h^2 \diff\mu_k} + U + \log Z \right|^s \diff\mu_k
\\
&\qquad \qquad
\leq \IntN h^2 \left| \log \frac{h^2}{\int h^2 \diff\mu_k}\right|^s \diff\mu_k + \IntN h^2 U^s \diff\mu_k
+ |\log Z|^s \IntN h^2 \diff\mu_k,
\end{aligned}
\end{equation}
where in the last inequality we used the fact that since $s\in (0,1]$ we have $(X+Y)^s \leq X^s + Y^s$ for $X,Y>0$. 

Before we start the usual procedure of applying Jensen's inequality, we note that the function $x \left|\log x \right|^s$ is bounded on $(0,1)$, and let 
$$D=\displaystyle\sup_{x\in (0,1)} x \left|\log x \right|^s < \infty.$$
Consider now the function $\log_+ x := \max\{0,\log x\}$. Then the above observation implies that
$$ x \left|\log x \right|^s \leq x(\log_+ x)^s + D \qquad \text{ for all } x>0.$$
With this in mind, we have
\begin{equation} \label{logsob2-2}
\begin{aligned}
\IntN h^2 \left| \log \frac{h^2}{\int h^2 \diff\mu_k} \right|^s \diff\mu_k
&= \IntN h^2 \diff\mu_k \cdot \IntN \frac{h^2}{\int h^2 \diff\mu_k} \left| \log \frac{h^2}{\int h^2 \diff\mu_k} \right|^s \diff\mu_k
\\
&\leq \IntN h^2 \diff\mu_k \cdot \left[ \IntN \frac{h^2}{\int h^2 \diff\mu_k} \left( \log_+ \frac{h^2}{\int h^2 \diff\mu_k} \right)^s \diff\mu_k + D\right].
\end{aligned}
\end{equation}

For the fixed function $h$ the measure $\frac{h^2}{\int h^2 \diff\mu_k} \diff\mu_k$ is a probability measure and so we can apply Jensen's inequality to the concave function $(\log_+ t)^s$ in the below as follows
\begin{align*}
\IntN \frac{h^2}{\int h^2 \diff\mu_k} \left( \log_+ \frac{h^2}{\int h^2 \diff\mu_k} \right)^s \diff\mu_k
&= \frac{1}{\delta^s}\IntN \frac{h^2}{\int h^2 \diff\mu_k} \left( \log_+ \left(\frac{h^2}{\int h^2 \diff\mu_k}\right)^\delta \right)^s \diff\mu_k
\\
&\leq \frac{1}{\delta^s} \left( \log_+ \IntN \left(\frac{h^2}{\int h^2 \diff\mu_k}\right)^{1+\delta} \diff\mu_k \right)^s 
\\
&= \frac{(1+\delta)^s}{\delta^s} \left( \log_+ \frac{\norm{h}_{2+2\delta}^2}{\norm{h}_2^2}  \right)^s.
\end{align*}

By standard elementary means we can show that there exists a decreasing function $\overline{C}(\epsilon)$ defined for all $\epsilon>0$ such that the inequality
$$ |\log_+ x|^s \leq \epsilon x + \overline{C}(\epsilon)$$
holds for all $x>0$. Applying this in the previous inequality, we obtain
\begin{equation} \label{logsob2-3}
\begin{aligned}
\IntN \frac{h^2}{\int h^2 \diff\mu_k} \left( \log_+ \frac{h^2}{\int h^2 \diff\mu_k} \right)^s \diff\mu_k
&\leq \frac{(1+\delta)^s}{\delta^s} \left( \epsilon \frac{\norm{h}_{2+2\delta}^2}{\norm{h}_2^2}  + \overline{C}(\epsilon) \right)
\\
&\leq \frac{(1+\delta)^s}{\delta^s} \frac{1}{\int h^2 \diff\mu_k} \left[ \epsilon C_{DS}^2 \norm{\nabla_k h}_2^2 + \overline{C}(\epsilon) \norm{h}_2^2 \right].
\end{aligned}
\end{equation}
Here in the last step we used the Sobolev inequality which holds if we choose $\delta>0$ such that $2+2\delta = \frac{2(N+2\gamma)}{N+2\gamma-2}$.

Next, we have
\begin{equation} \label{logsob2-4}
\norm{\nabla_k h}_2^2 = \frac{1}{Z} \IntN |\nabla_k(fe^{-U/2})|^2 \diff\mu_k 
\leq 2 \IntN |\nabla_k f|^2 \diff\mu_U + \frac{1}{2} \IntN f^2 |\nabla U|^2 \diff\mu_U.
\end{equation}

Combining (\ref{logsob2-1}), (\ref{logsob2-2}), (\ref{logsob2-3}) and (\ref{logsob2-4}), we have
\begin{align*}
\IntN f^2 \left| \log \frac{f^2}{\int f^2 \diff\mu_U} \right|^s \diff\mu_U
&\leq \left(\frac{N+2\gamma}{2}\right)^s 2C_{DS}^2\epsilon  \IntN |\nabla_k f|^2 \diff\mu_U 
\\
&\qquad
+ \left[\left(\frac{N+2\gamma}{2}\right)^s \overline{C}(\epsilon) + |\log Z|^s \right] \IntN f^2 \diff\mu_U
\\
&\qquad
+ \left(\frac{N+2\gamma}{2}\right)^s \frac{C_{DS}^2}{2}\epsilon \IntN f^2 |\nabla U|^2 \diff\mu_U
+ \IntN f^2 U^s \diff\mu_U.
\end{align*}
But $|\nabla U|= p |x|^{p-1}$ and $U^s = |x|^{2(p-1)}$, so the last two terms can be computed as follows
$$ \IntN f^2 U^s \diff\mu_U = \IntN f^2 |x|^{2(p-1)} \diff\mu_U = \frac{1}{p^2} \IntN f^2 |\nabla U|^2 \diff\mu_U.$$
Finally, from Proposition \ref{Ubounds-2p-prop} and a relabelling of the constants, we obtain
$$ \IntN f^2 \left| \log \frac{f^2}{\int f^2 \diff\mu_U} \right|^s \diff\mu_U
\leq \epsilon \IntN |\nabla_k f|^2 \diff\mu_U + C(\epsilon) \IntN f^2 \diff\mu_U,$$
where $C(\epsilon) = A_1 \overline{C}(a\epsilon) + A_2 \epsilon + A_3$, for some constants $A_1, A_2,A_3, a>0$. Choosing small $\epsilon >0$ we have $C(\epsilon)>0$. This concludes the proof.
\end{proof}

\begin{thm}
Let $U(x)=|x|^p$ with $1\leq p <\infty$ and consider the finite measure $\diff\mu_U = e^{-U} \diff\mu_k$. Let $s=\frac{p-1}{p}$. Then there exist some constants $C_1,C_2>0$ such that for all $f\in W^{1,1}_k(\mu_U)$, we have the inequality
$$ \IntN f \left|\log \frac{|f|}{\int |f| \diff\mu_U} \right|^s \diff\mu_U \leq C_1 \IntN |\nabla_k f| \diff\mu_U + C_2 \IntN |f| \diff\mu_U.$$
\end{thm}

\begin{proof}
The proof is similar to that of the previous result except for in this case we rely on the Sobolev inequality
$$ \norm{h}_q \leq C \norm{\nabla_k h}_1$$
where $q=\frac{N+2\gamma}{N+2\gamma-1}$, and the $U$-bound of Proposition \ref{Ubounds-p}.
\end{proof}

\section{Poincar\'e inequalities} \label{poincaresection}

In this section we discuss Poincar\'e inequalities for Dunkl operators. These will be used in the next section to improve some of our previous log-Sobolev inequalities, but are also of independent interest. 

We note that by solving the equation $C(\epsilon)=0$, Theorem \ref{logsobolev} implies the following tight log-Sobolev inequality
\begin{equation} \label{tightlogsobolev}
\IntN f^2 \log \frac{f^2}{\int f^2 \diff\mu_k} \diff\mu_k \leq C \IntN |\nabla_k f|^2 \diff\mu_k,
\end{equation}
which holds for a constant $C>0$ and for all $f\in H^1_k(\RR^N)$. 

Using a standard argument (see for example \cite{BGL}), this inequality implies the following Poincar\'e inequality.

\begin{thm} \label{Dunklpoincare}
Let $R>0$ and consider the ball $B_R:= \{ x\in\RR^N : |x| \leq R\}$. There exists a constant $C>0$ independent of $R$ such that for any $f\in H^1_k(B_R, \mu_k)$ we have the inequality
$$ \int_{B_R} \left| f- \frac{1}{\mu_k(B_R)} \int_{B_R} f \diff\mu_k\right|^2 \diff\mu_k 
\leq C R^2 \int_{B_R} |\nabla_k f|^2 \diff\mu_k.$$
\end{thm}

\begin{proof}
To simplify the notation, let $\tilde{\mu}_R := \frac{1}{\mu_k(B_R)} \mu_k$ be the Dunkl probability measure on the ball $B_R$. Note that it is enough to prove the Theorem for $f$ that satisfies the additional assumption $\int f \diff\tilde{\mu}_R =0$, in which case the inequality takes the form
\begin{equation} \label{poincare0}
\int f^2 \diff\tilde{\mu}_R \leq C R^2 \int |\nabla_k f|^2 \diff\tilde{\mu}_R.
\end{equation}
To obtain the general case it is then enough to replace $f$ by $f-\int f \diff\tilde{\mu}_R$ in \eqref{poincare0}.

For any $\epsilon >0$ consider the function $g=1+\epsilon f$. A Taylor expansion shows that 
$$ g^2 \log \frac{g^2}{\int g^2 \diff\tilde{\mu}_R} = 2 \epsilon f + 3\epsilon^2 f^2 - \epsilon^2 \int f^2 \diff\tilde{\mu}_R + o(\epsilon^2),$$
and thus
\begin{equation} \label{poincaretaylor}
\int g^2 \log \frac{g^2}{\int g^2 \diff\tilde{\mu}_R} \diff\tilde{\mu}_R 
= 2\epsilon^2 \int f^2 \diff\tilde{\mu}_R + o(\epsilon^2),
\end{equation}
as $\epsilon \to 0$.

From Theorem \ref{logsobolev} we have that 
$$ \int g^2 \log \frac{g^2}{\int g^2 \diff\tilde{\mu}_R} \diff\tilde{\mu}_R \leq \delta \int |\nabla_k g|^2 \diff\tilde{\mu}_R + (C(\delta)+ \log (\mu_k(R)) \int g^2 \diff\tilde{\mu}_R,$$
holds for all $\delta>0$. However, using the fact that $\mu_k(B_R)= c' R^{N+2\gamma}$ for a constant $c'>0$, we find that $\delta = c'' R^2$, for a constant $c''>0$, solves the equation $C(\delta)+\log(\mu_k(B_R))=0$. Therefore, we have the tight log-Sobolev inequality
\begin{equation} \label{tightlogsobolevball}
 \int g^2 \log \frac{g^2}{\int g^2 \diff\tilde{\mu}_R} \diff\tilde{\mu}_R
 \leq c'' R^2 \int |\nabla_k g|^2 \diff\tilde{\mu}_R.
\end{equation}

Combining \eqref{poincaretaylor} and \eqref{tightlogsobolevball}, and letting $\epsilon \to 0$, we have obtained \eqref{poincare0}, as required. 
\end{proof}

\begin{remark}
This Poincar\'e inequality corresponds to the classical Neumann-Poincar\'e inequality. A Dirichlet-Poincar\'e inequality for Dunkl operators was also proved in \cite{Vel19}. Namely, we have the result:
$$ \int_\Omega |f|^2 \diff\mu_k \leq C(\Omega) \int_\Omega |\nabla_k f|^2 \diff\mu_k,$$
which holds on any bounded domain $\Omega \subset \RR^N$ for a constant $C(\Omega)>0$ and for all $f\in C_0^\infty (\RR^N)$.
\end{remark}

We can now use the previous result together with the $U$-bounds proved above to obtain a Poincar\'e inequality for the weighted measure $\mu_U$.

\begin{prop} \label{weighteddunklpoincare}
Let $p>1$ and consider the weighted probability measure $\diff\mu_U = \frac{1}{Z} e^{-U}\diff\mu_k$, where $U(x)=|x|^p$ and $Z=\IntN e^{-U} \diff\mu_k$. We then have the Poincar\'e inequality
\begin{equation} \label{weighteddunklpoincare_eqn} 
\IntN \left| f-\IntN f \diff\mu_U \right|^2 \diff\mu_U \leq C \IntN |\nabla_k f|^2 \diff\mu_U
\end{equation}
which holds for all functions $f \in H^1_k(\mu_U)$, with a constant $C>0$ independent of $f$.
\end{prop}

\begin{proof}
We first note that for any constant $\zeta \in \RR$ we have
\begin{align*}
\left| \IntN f \diff\mu_U - \zeta \right| 
\leq  \IntN |f-\zeta| \diff\mu_U
\leq \left( \IntN |f-\zeta|^2 \diff\mu_U \right)^\frac{1}{2},
\end{align*}
so by the triangle inequality for the $L^2(\mu_U)$ norm we obtain
\begin{equation} \label{poincare_step}
\IntN \left| f-\IntN f \diff\mu_U \right|^2 \diff\mu_U \leq 4 \IntN |f-\zeta|^2 \diff\mu_U.
\end{equation}
Thus, it is enough to prove the inequality
\begin{equation} \label{weightedpoincarestep1} 
\IntN |f-\zeta|^2 \diff\mu_U \leq C \IntN |\nabla_k f|^2 \diff\mu_U
\end{equation}
for some $\zeta\in\RR$.

Let $R>0$ and let $B_R=\{ |x| \leq R\}$. We will prove (\ref{weightedpoincarestep1}) with $\zeta = \frac{1}{\mu_k(B_R)}\int_{B_R} f\diff\mu_k$ for large enough $R$. Firstly, we have
\begin{equation} \label{Dunklweightedpoincare1}
\begin{aligned}
\int_{B_R} |f-\zeta|^2 \diff\mu_U 
&\leq \frac{1}{Z} \int_{B_R} \left| f- \frac{1}{\mu_k(B_R)} \int_{B_R} f \diff\mu_k\right|^2 \diff\mu_k 
\\
&\leq \frac{C}{Z} R^2 \int_{B_R} |\nabla_k f|^2 \diff\mu_k
\\
&\leq C R^2 e^{R^p} \int_{B_R} |\nabla_k f|^2 \diff\mu_U.
\end{aligned}
\end{equation}
Here we used the Poincar\'e inequality of Theorem \ref{Dunklpoincare} and the bounds $e^{-R^p} \leq e^{-U} \leq 1$ on $B_R$. 

On the other hand, we can use the $U$-bounds of Proposition \ref{Ubounds-2p-prop} to the remaining integral as follows
\begin{align*}
\int_{\RR^N \setminus B_R} |f-\zeta|^2 \diff\mu_U 
&\leq R^{-2(p-1)} \int_{|x| \geq R} |f(x)-\zeta|^2 |x|^{2(p-1)} \diff\mu_U 
\\
&\leq C R^{-2(p-1)} \int_{|x|\geq R} |\nabla_k f|^2 \diff\mu_U + DR^{-2(p-1)} \int_{|x|\geq R} |f-\zeta|^2 \diff\mu_U.
\end{align*} 
But $R$ was an arbitrary positive number so we are free to choose it such that $DR^{-2(p-1)}<1$. Then we have
\begin{equation} \label{Dunklweightedpoincare2}
\int_{\RR^N \setminus B_R} |f-\zeta|^2 \diff\mu_U \leq \frac{CR^{-2(p-1)}}{1-DR^{-2(p-1)}} \int_{|x|\geq R} |\nabla_k f|^2 \diff\mu_U.
\end{equation}
Adding the inequalities (\ref{Dunklweightedpoincare1}) and (\ref{Dunklweightedpoincare2}), we obtain (\ref{weightedpoincarestep1}), and therefore, by the observation above, the Proposition is proved.
\end{proof}

\section{Tight log-Sobolev inequalities} \label{SEC:tightlogsob}

We now have all the ingredients to obtain tight log-Sobolev inequalities. The first is a tight version of the log-Sobolev inequality from Theorem \ref{logsob>2_thm}.

\begin{thm} \label{tightlogsobolev1}
Let $U(x)=|x|^p$ for $p\geq 2$ and consider the probability measure $\diff\mu_U = \frac{1}{Z} e^{-U} \diff\mu_k$, where $Z=\IntN e^{-U} \diff\mu_k$. Then there exists a constant $C>0$ such that the inequality
\begin{equation} \label{tightlogsob>2}
\IntN f^2 \log \frac{f^2}{\int f^2 \diff\mu_U} \diff\mu_U \leq C \IntN |\nabla_k f|^2 \diff\mu_U
\end{equation}
holds for all $f\in H^1_k(\mu_U)$.
\end{thm}

In order to prove this result we will need the following inequality, known as Rothaus's lemma.

\begin{lem} [\cite{Rothaus}]
Recall that 
$$ \mathrm{Ent}(g) := \IntN g \log g \diff\mu_U - \IntN g \diff\mu_U \log \IntN g \diff\mu_U,$$
for $g\geq 0$. Then, for all $f$ measurable with $\IntN f \diff\mu_U = 0$, we have the inequality
$$ \mathrm{Ent} ((f+c)^2) \leq \mathrm{Ent}(f^2) + 2\IntN f^2 \diff\mu_U,$$
for all $c \in \RR$.
\end{lem}

\begin{proof} [Proof of Theorem \ref{tightlogsobolev1}]
By Rothaus's lemma we have
$$ \mathrm{Ent}(f^2) \leq \mathrm{Ent}\left(\left(f-\int f \diff\mu_U\right)^2\right) + 2 \IntN \left(f-\int f \diff\mu_U\right)^2 \diff\mu_U.$$
Furthermore, from Theorem \ref{logsob>2_thm}, we have
$$ \mathrm{Ent}(f^2) \leq C_1 \IntN |\nabla_k f|^2 \diff\mu_U + (2+C_2) \IntN \left(f-\int f \diff\mu_U\right)^2 \diff\mu_U.$$
Finally, using the Poincar\'e inequality of Proposition  \ref{weighteddunklpoincare}, we obtain
$$ \mathrm{Ent}(f^2) \leq (C_1+C(2+C_2)) \IntN |\nabla_k f|^2 \diff\mu_U,$$
which is exactly what we wanted to prove.
\end{proof}

As we shall see in the next section, the condition $p\geq 2$ in the previous Theorem is necessary. However, in the range $1< p <2$ we can still obtain a $\Phi$-Sobolev inequality. This is the object of the following theorem, which is a tight version of the generalised log-Sobolev inequality of Theorem \ref{weightedlogsobolev1}, and it is obtained from this result in a manner very similar to the proof that we have just seen. 

\begin{thm} \label{tightlogsobolev2}
Let $1< p < 2$ and $s=2\frac{p-1}{p}$. Let $U(x)=|x|^p$ and consider the probability measure $\diff\mu_U = \frac{1}{Z} e^{-U} \diff\mu_k$, where $Z=\IntN e^{-U} \diff\mu_k$. Let also 
$$ \Phi(x) = x (\log(x+1))^s.$$
Then there exists a constant $C>0$ such that the inequality
$$ \IntN \Phi(f^2) \diff\mu_U - \Phi \left( \IntN f^2 \diff\mu_U \right)
\leq C \IntN |\nabla_k f|^2 \diff\mu_U,$$
holds for all $f\in H^1_k(\mu_U)$.
\end{thm}

As before, we need the following generalisation of Rothaus's lemma.

\begin{lem} [\cite{LZ}]
Let $\Phi$ be as in the statement of the Theorem and define, for $g\geq 0$, 
$$ \mathrm{Ent}_\Phi (g) := \IntN \Phi(g) \diff\mu_U - \Phi \left( \IntN g \diff\mu_U \right).$$
Then there exist constants $A_1, B_1>0$ such that for any $f$ with $\IntN f \diff\mu_U =0 $ we have
$$ \mathrm{Ent}_\Phi ((f+c)^2) \leq A_1 \mathrm{Ent}_\Phi(f^2) + B_1 \IntN f^2 \diff\mu_U$$
for all $c\in \RR$.
\end{lem}

\begin{proof} [Proof of Theorem \ref{tightlogsobolev2}]
The proof of this goes along the same lines as the proof of Theorem \ref{tightlogsobolev1} although in this case we cannot apply Theorem \ref{weightedlogsobolev1} directly. Instead, we note that
\begin{align*}
\mathrm{Ent}_\Phi(g) 
&= \IntN g \left[\left(\log (1+g) \right)^s - \left( \log \left(1+\IntN g \right)\right)^s \right] \diff\mu_U
\\
&\leq \IntN g \left| \log \frac{g+1}{\int g \diff\mu_U +1}\right|^s \diff\mu_U,
\end{align*}
where we used the inequality $(X+Y)^s \leq X^s + Y^s$ which holds for all $X,Y \geq 0$ since $s\in [0,1]$. We compute the integral on the right hand side separately over $X=\left\{x : g(x) \geq \IntN g \diff\mu_U \right\}$ and $\overline{X}=\RR^N\setminus X$. On $X$ we have 
$$ 1 \leq \frac{g+1}{\int g \diff\mu_U +1} \leq \frac{g}{\int g \diff\mu_U},$$
so
$$ \int_X g \left| \log \frac{g+1}{\int g \diff\mu_U +1}\right|^s \diff\mu_U
\leq \IntN g \left| \log \frac{g}{\int g \diff\mu_U}\right|^s \diff\mu_U.$$
On the other hand, on $\overline{X}$ we have
$$ 1 \leq \frac{\int g \diff\mu_U +1}{g+1} \leq 1 + \frac{\int g \diff\mu_U}{g},$$
so 
\begin{align*}
\int_{\overline{X}} g \left| \log \frac{g+1}{\int g \diff\mu_U +1}\right|^s \diff\mu_U 
&= \int_{\overline{X}} g \left( \log \frac{\int g \diff\mu_U+1}{g +1}\right)^s \diff\mu_U
\\
&\leq \int_{\overline{X}} g \left(\frac{\int g \diff\mu_U}{g}\right)^s \diff\mu_U
\leq \IntN g \diff\mu_U,
\end{align*}
where we first used the inequality $\log(1+x) \leq x$ for all $x\geq 0$, and then the fact that $s\leq 1$, so $\left(\frac{\int g \diff\mu_U}{g}\right)^s \leq \frac{\int g \diff\mu_U}{g}$. Thus
$$ \mathrm{Ent}_\Phi(g) \leq \IntN g \left| \log \frac{g}{\int g \diff\mu_U}\right|^s \diff\mu_U + \IntN g \diff\mu_U.$$

We can now apply the same strategy as before. First, by Theorem \ref{weightedlogsobolev1}, we have
$$ \mathrm{Ent}_\Phi(g^2) \leq C_1 \IntN |\nabla_k g|^2 \diff\mu_U + (C_2+1) \IntN g^2 \diff\mu_U.$$
Taking $g=f-\IntN f \diff\mu_U$ in this inequality and applying the previous Lemma, we have
$$ \mathrm{Ent}_\Phi(f^2) \leq A_1 C_1 \IntN |\nabla_k f|^2 \diff\mu_U + (A_1(C_2+1) + B_1) \IntN \left(f -\IntN f \diff\mu_U \right)^2 \diff\mu_U.$$
Finally, using the Poincar\'e inequality of Proposition \ref{weighteddunklpoincare}, the proof is complete.
\end{proof}

\section{Applications}
\label{SEC:appl}

\subsection{Exponential integrability and measure concentration}

As a consequence of the tight log-Sobolev inequality of Theorem \ref{tightlogsobolev1}, we can prove exponential integrability for Lipschitz functions. The proof of this fact uses the classical Herbst argument (see \cite{BGL}); for completeness, we give a sketch of the argument here.

\begin{thm} \label{expint_thm}
Let $U(x)=|x|^p$ for $p\geq 2$ and consider the probability measure $\diff\mu_U = \frac{1}{Z} e^{-U}\diff \mu_k$. For any $a$-Lipschitz function $f$ and for any $b<\sqrt{\frac{2}{aC}}$ (where $C$ is the constant in \eqref{tightlogsob>2}) we have
$$ \IntN e^{b^2f^2/2} \diff\mu_U <\infty.$$
\end{thm}

\begin{proof}
We first prove that if $f$ is $G$-invariant (in addition to the assumptions above), then for any $s\in \RR$ we have
\begin{equation} \label{expint_step}
\IntN e^{sf} \diff\mu_U \leq \exp\left(s \IntN f \diff\mu_U + aC\frac{s^2}{4} \right).
\end{equation}
It is enough to prove this inequality for a bounded function $f$. Indeed, the general case can then be obtained by defining $f_n(x) = \max\{\min\{f(x),n\}, -n\}$ for all $n\in \NN$, and taking the limit $n\to \infty$ in \eqref{expint_step} using Fatou's lemma.

From inequality \eqref{tightlogsob>2} applied to the function $e^{sf/2}$ (recall that $f$ is $G$-invariant) we obtain
\begin{equation} \label{expint_logsob}
\IntN e^{sf} \log e^{sf} \diff\mu_U - \IntN e^{sf} \diff\mu_U \log \IntN e^{sf} \diff\mu_U \leq C\frac{s^2}{4} \IntN e^{sf} |\nabla f|^2 \diff\mu_U.
\end{equation}
Define $X (s) = \IntN e^{sf} \diff\mu_U$ and hence $ X'(s) = s \IntN f e^{sf} \diff\mu_U$. Using this new notation, inequality \eqref{expint_logsob} becomes
$$ s X'(s) - X(s) \log X(s) \leq aC \frac{s^2}{4} X(s).$$
Here we also used the fact that since $f$ is $a$-Lipschitz, then $|\nabla f| \leq a$ a.e. Letting $Y (s) = \frac{1}{s} \log X(s)$ (with $Y(0) = \int f \diff\mu_U$), this further becomes
$$ Y'(s) \leq \frac{aC}{4}.$$
Integrating this inequality we obtain \eqref{expint_step}.

Multiplying \eqref{expint_step} with $e^{-s^2/(2b^2)}$ we obtain
$$ \int_{-\infty}^\infty\IntN e^{-\frac{s^2}{2b^2}+sf} \diff\mu_U \diff s
\leq \int_{-\infty}^\infty e^{-\frac{s^2}{2b^2} + aC \frac{s^2}{4} + s \int f \diff\mu_U} \diff s.$$
Using Fubini's theorem and computing the integrals with respect to $s$, it follows that
$$ \IntN e^{b^2f^2/2} \diff\mu_U \leq \frac{\sqrt{2}}{\sqrt{2-b^2aC}} \exp \left(\frac{c^2}{2-b^2aC} \left(\IntN f \diff\mu_U \right)^2\right).$$
To conclude the proof in this case, it is enough to check that $f$ is integrable. We refer to the proof of \cite[Proposition 4.4.2]{BGL} for a discussion of this using the Poincar\'e inequality. 

Finally, consider the general case when $f$ is not necessarily $G$-invariant. For any Weyl chamber $H$, let $\restr{f}{H}:H\to \RR$ be the restriction of $f$ to $H$, and let $\tilde{f}_H:\RR^N \to \RR$ be the $G$-invariant function equal to $\restr{f}{H}$ on each Weyl chamber, i.e.,
$$ \tilde{f}_H(\sigma_\alpha x) = \restr{f}{H} (x) \qquad \forall x\in H,\; \forall \alpha\in R_+.$$
Then $\tilde{f}_H$ is also $a$-Lipschitz and so, from the above, we have
$$ |G| \int_H e^{b^2 f^2/2} \diff\mu_U = \IntN e^{b^2\tilde{f}_H^2} \diff\mu_U < \infty.$$
Here in the first equality we used property \eqref{property_Ginv} and the fact that $\tilde{f}_H = f$ on $H$. Finally, we have
$$ \IntN e^{b^2 f^2/2} \diff\mu_U = \sum_H \int_H e^{b^2 f^2/2} \diff\mu_U <\infty,$$
where the sum goes over all the Weyl chambers $H$. This completes the proof.
\end{proof}

As a by-product of this proof, we next obtain a Gaussian measure concentration property.

\begin{cor}
Let $U(x)=|x|^p$ for $p\geq 2$ and consider the probability measure $\diff\mu_U = \frac{1}{Z} e^{-U}\diff \mu_k$. For any $G$-invariant $a$-Lipschitz function $f$ and for any $r\geq 0$ we have
\begin{equation} \label{measconc_eqn} 
\mu_U \left( f\geq \IntN f\diff\mu_U + r \right)
\leq e^{-r^2/(aC)}.
\end{equation}
\end{cor}

\begin{proof}
By Markov's inequality and \eqref{expint_step} we have, for any $s\in \RR$, 
\begin{align*}
\mu_U \left( f\geq \IntN f\diff\mu_U + r \right) 
&= \mu_U \left( e^{sf} \geq \exp\left(s \IntN f\diff\mu_U + sr\right) \right)
\\
&\leq e^{-s\int f \diff\mu_U -sr} \IntN e^{sf} \diff\mu_U
\leq e^{-sr + aC \frac{s^2}{4}}.
\end{align*}
The right hand side is minimised for $s=\frac{2r}{aC}$, and replacing this in the above inequality we obtain exactly \eqref{measconc_eqn}, as required.
\end{proof}

\begin{remark}
Using exponential integrability we can see that the condition $p\geq 2$ in Theorem \ref{tightlogsobolev1} is necessary. Indeed, assuming that \eqref{logsob>2} holds for some $1<p<2$, then Theorem \ref{expint_thm} can be extended in exactly the same way to this case. In other words, this shows that $e^{b^2f^2/2}$ is integrable if $f$ is a Lipschitz function. In particular, taking $f(x)=|x|$ which is $1$-Lipschitz, we have that
$$ \IntN e^{-|x|^p+b^2|x|^2/2} \diff\mu_k <\infty,$$
for some $b>0$, which contradicts the assumption $p<2$. 
\end{remark}
\subsection{Functional inequalities for singular Boltzmann-Gibbs measures}

As discussed in the introduction, the Dunkl setting allows us to rephrase some functional inequalities related to Boltzmann-Gibbs measures. We exploit this connection here to obtain such applications. The inequalities in this subsection are all stated for the classical gradient $\nabla f$, and the probability measures we consider are supported on the closure of a Weyl chamber $H$, and take the form
\begin{equation} \label{mu_UH}
\diff \mu_{U,H} = \frac{1}{Z_H} \1_{H} e^{-|x|^p} \diff\mu_k,
\end{equation}
where $Z_H=\IntN e^{-|x|^p} \1_H \diff\mu_k$ is a normalising constant, $1_H$ is the indicator function of any Weyl chamber, and $p>1$. 

Firstly, as a corollary of Proposition \ref{weighteddunklpoincare}, we obtain a Poincar\'e inequality for this setting. 

\begin{thm}
Let $p>1$. Let $H$ be any Weyl chamber associated with the root system $R$ and consider the probability measure $\diff\mu_{U,H}$ defined by \eqref{mu_UH}.  We then have the Poincar\'e inequality
\begin{equation} \label{poincare_H_eqn} 
\IntN \left| f-\IntN f \diff\mu_{U,H} \right|^2 \diff\mu_{U,H} \leq \tilde{C} \IntN |\nabla f|^2 \diff\mu_{U,H}
\end{equation}
which holds for all functions $f \in H^1(\mu_{U,H})$, with a constant $\tilde{C}>0$ independent of $f$.
\end{thm}

\begin{proof}
Similarly to \eqref{poincare_step} we obtain
\begin{equation} \label{poincare_step_H}
\IntN \left| f - \IntN f \diff\mu_{U,H} \right|^2 \diff\mu_{U,H}
\leq 4 \IntN |f-\zeta|^2 \diff\mu_{U,H}
\end{equation}
for any $\zeta \in \RR$. 

Let $\restr{f}{H}:H\to \RR$ be the restriction of $f$ to $H$, and let $\tilde{f}:\RR^N \to \RR$ be the $G$-invariant function equal to $\restr{f}{H}$ on each Weyl chamber, i.e.,
\begin{equation} \label{defn_Ginvfct}
\tilde{f}(\sigma_\alpha x) = \restr{f}{H} (x) \qquad \forall x\in H,\; \forall \alpha\in R_+.
\end{equation}
Applying the Poincar\'e inequality \eqref{weighteddunklpoincare_eqn} to the function $\tilde{f}$ we obtain
\begin{equation} \label{poincare_H_step2} 
\IntN \left| \tilde{f}-\IntN \tilde{f} \diff\mu_U \right|^2 \diff\mu_U \leq C \IntN |\nabla_k \tilde{f}|^2 \diff\mu_U.
\end{equation}

Let us note here that for any $G$-invariant function $g$ and any Weyl chamber $H$ we have
\begin{equation} \label{property_Ginv}
\IntN g \diff\mu_k = \sum_{H'} \int_{H'} g \diff\mu_k = |G| \int_{H} g \diff\mu_k,
\end{equation}
where the sum goes over all Weyl chambers $H'$ and recall that $|G|$ is the number of Weyl chambers. Indeed, this is because for any Weyl chamber $H'$ there exists $\alpha\in R$ such that $H'=\sigma_\alpha H$, so by a change of variables $y=\sigma_\alpha x$ we have
$$ \int_{H} g \diff\mu_k = \int_{H'} g \diff\mu_k.$$

Since $\tilde{f}$ is $G$-invariant, we have $\nabla_k f=\nabla f$, and using property \eqref{property_Ginv}, the inequality \eqref{poincare_H_step2} becomes
$$\int_H \left|f-|G| \int_H f \diff\mu_U \right|^2 \diff\mu_U \leq C \int_H |\nabla f|^2 \diff\mu_U.$$
Using now the fact that $\1_H \diff\mu_U = \frac{Z_H}{Z} \diff\mu_{U,H}$, this inequality becomes
$$ \IntN \left| f - |G| \frac{Z_H}{Z} \IntN f \diff\mu_{U,H}\right|^2 \diff\mu_{U,H}
\leq C \IntN |\nabla f|^2 \diff\mu_{U,H}.$$
Taking $\zeta = |G| \frac{Z_H}{Z} \IntN f \diff\mu_{U,H}$ in \eqref{poincare_step_H} together with the previous inequality imply \eqref{poincare_H_eqn} with $\tilde{C}=4C$.
\end{proof}

Similarly, from Theorem \ref{tightlogsobolev1} we obtain a tight log-Sobolev inequality for this setting when $p\geq 2$. 

\begin{thm} \label{logsob_appl}
Let $p\geq 2$. Let $H$ be any Weyl chamber associated with the root system $R$ and consider the probability measure $\diff\mu_{U,H}$ defined by \eqref{mu_UH}. Then there exists a constant $C>0$ such that the inequality 
\begin{equation}
\IntN f^2 \log \frac{f^2}{\int f^2 \diff\mu_{U,H}} \diff\mu_{U,H}
\leq C \IntN |\nabla f|^2 \diff\mu_{U,H}
\end{equation}
holds for all $f\in H^1(\mu_{U,H})$.
\end{thm}

\begin{proof}
Consider the $G$-invariant function $\tilde{f}: \RR^N \to \RR$ defined by \eqref{defn_Ginvfct}. Applying the log-Sobolev inequality \eqref{tightlogsob>2} to the function $\tilde{f}$ and using property \eqref{property_Ginv}, we obtain
$$ \int_H f^2 \log \frac{f^2}{\int_H f^2 \diff\mu_U} \diff\mu_U
\leq C \int_H |\nabla f|^2 \diff\mu_U + \log |G| \int_H f^2 \diff\mu_U.$$
Using now the fact that $\1_H \diff\mu_U = \frac{Z_H}{Z} \diff\mu_{U,H}$, this inequality becomes
$$ \IntN f^2 \log \frac{f^2}{\int f^2 \diff\mu_{U,H}} \diff\mu_{U,H}
\leq C \IntN |\nabla f|^2 \diff\mu_{U,H} + \log \left(|G| \frac{Z_H}{Z}\right) \IntN f^2 \diff\mu_{U,H}.$$
To obtain a tight log-Sobolev inequality we use the same method as in the proof of Theorem \ref{tightlogsobolev1}, making use of the Rothaus lemma and the Poincar\'e inequality \eqref{poincare_H_eqn}.
\end{proof}

\begin{example}
Let us consider the case of root system $A_{N-1}$ where we have $R_+ = \{ e_i-e_j | 1\leq i < j \leq N\}$ and one choice of Weyl chamber is $H=\{ x\in \RR^N | x_1 > x_2 > \ldots >x_N \}$. In this case, all roots belong to the same orbit of the reflection group $G=S_N$, so the multiplicity function reduces to a constant, i.e., $k_\alpha = k \geq 0$ for all $\alpha\in R_+$, and $w_k(x)=\prod_{i<j} (x_i-x_j)^{2k}$. Thus, the measure  $\mu_{U,H}$ becomes
$$ \diff \mu_{U,H} = \frac{1}{Z_H} \1_{\{x_1>x_2>\ldots >x_N\}} e^{-|x|^p} \prod_{1\leq i <j \leq N} (x_i-x_j)^{2k}.$$
\end{example}

\begin{example}
In the case of root system $B_N$ we have $R_+ = \{\sqrt{2}e_i | 1 \leq i \leq N \} \cup \{e_i\pm e_j | 1\leq i < j \leq N \}$ and a choice of Weyl chamber is $H=\{ x\in \RR^N | x_1 >x_2 > \ldots >x_N >0\}$. Here, the multiplicity function reduces to two constants, say $k_1,k_2 \geq 0$ (depending on whether the root is of the form $\sqrt{2}e_i$, or $e_i\pm e_j$) and the Dunkl weight becomes $w_k(x) = 2^{k_1N} \prod_{i=1}^N |x_i|^{2k_1} \prod_{i<j} (x_i-x_j)^{2k_2} \prod_{i<j} (x_i+x_j)^{2k_2}$. Thus, the measure \eqref{mu_UH} in this case equals
$$ \diff\mu_{U,H} = \frac{1}{Z_H} \1_{\{x_1>\ldots >x_N >0\}} e^{-|x|^p} \prod_{i=1}^N |x_i|^{2k_1} \prod_{i<j} (x_i-x_j)^{2k_2} \prod_{i<j} (x_i+x_j)^{2k_2}.$$
\end{example}

Finally, we note that we can obtain a $\Phi$-Sobolev inequality in the range $1<p<2$ which complements Theorem \ref{logsob_appl}. The proof of this fact uses Theorem \ref{tightlogsobolev2} and goes along the same lines as above so we omit it here. 

\begin{thm}
Let $1<p<2$ and $s=2\frac{p-1}{p}$. Let $H$ be any Weyl chamber associated with the root system $R$ and consider the probability measure $\diff\mu_{U,H}$ defined by \eqref{mu_UH}. Define 
$$ \Phi(x) = x (\log(x+1))^s.$$
Then there exists a constant $C>0$ such that the inequality 
\begin{equation}
\IntN \Phi(f^2) \diff\mu_{U,H} - \Phi \left( \IntN f^2 \diff\mu_{U,H} \right)
\leq C \IntN |\nabla f|^2 \diff\mu_{U,H}
\end{equation}
holds for all $f\in H^1(\mu_{U,H})$.
\end{thm}

\noindent \textbf{Acknowledgements.} The author wishes to thank Boguslaw Zegarlinski for introducing him to the problem and for useful advice. Financial support from EPSRC is also gratefully acknowledged. 

%\nocite{*}
%\pagebreak

\bibliographystyle{plain}
\bibliography{ref}

\end{document}